\documentclass{amsart}
\usepackage[T1]{fontenc}
\usepackage{amsfonts}
\usepackage{mathrsfs}
\usepackage{mathtools}
\usepackage{amssymb}
\usepackage[colorlinks=true,linkcolor=blue,allcolors=blue
]{hyperref}
\usepackage[noabbrev,capitalize,nameinlink]{cleveref}
\usepackage{amsmath,blkarray,booktabs,bigstrut}
\usepackage{cite}
\usepackage{graphicx}
\usepackage{dutchcal}
\numberwithin{equation}{section}

\DeclareMathOperator{\supp}{supp}
\DeclareMathOperator{\spanning}{span}
\theoremstyle{plain}
  \newtheorem{theorem}{Theorem}[section]
  \newtheorem{proposition}[theorem]{Proposition}
  \newtheorem{lemma}[theorem]{Lemma}
  \newtheorem{corollary}[theorem]{Corollary}

  \newtheorem{remark}[theorem]{Remark}

 \theoremstyle{definition}
   \newtheorem{definition}[theorem]{Definition}

\begin{document}
\include{psfig}
\title[Paley-Wiener Theorem for Probabilistic Frames]{Paley-Wiener Theorem for Probabilistic Frames}

\author{Dongwei Chen}

\begin{abstract}
This paper establishes Paley-Wiener perturbation theorems for probabilistic frames.
The classical Paley-Wiener perturbation theorem shows that if a sequence is close to a basis in a Banach space, then this sequence is also a basis.  Similar perturbation results have been established for frames in Hilbert spaces. In this work, we show that if a probability measure is sufficiently close to a probabilistic frame in an appropriate sense, then this probability measure is also a probabilistic frame. Moreover, we obtain explicit frame bounds for such probability measures that are close to a given probabilistic frame in the $2$-Wasserstein metric. This yields an alternative proof of the fact that the set of probabilistic frames is open in $\mathcal{P}_2(\mathbb{R}^n)$ under the $2$-Wasserstein topology.
\end{abstract}

\address{Department of Mathematics, Colorado State University, Fort Collins, CO, USA, 80523.}
\email{dongwei.chen@colostate.edu}
\maketitle

\section{Introduction and Main Results}

The Paley-Wiener perturbation theorem, introduced by Raymond E. A. C. Paley and Norbert Wiener, is a classical result concerning the stability of bases in Hilbert spaces \cite{paley1934fourier}. 
Ralph P. Boas later noticed that Paley and Wiener's proof also held in Banach spaces:

\begin{theorem}[Theorem 1 in \cite{boas1940general}] \label{PaleyBasis}
    Let $\{{\bf x}_i\}_{i=1}^\infty$ be a basis for a Banach space $\mathcal{X}$ with norm $ \Vert \cdot \Vert $ and $\{{\bf y}_i\}_{i=1}^\infty$ a sequence in $\mathcal{X}$. If there exists $0 \leq \lambda <1$  such that 
    \begin{equation*}
     \Big   \Vert \sum_{i=1}^n c_i ({\bf x}_i-{\bf y}_i)  \Big \Vert \leq \lambda  \Big \Vert \sum_{i=1}^n c_i {\bf x}_i \Big \Vert
    \end{equation*}
for all scalars $c_1, \cdots, c_n \ (n =1, 2, \cdots)$,  then $\{{\bf y}_i\}_{i=1}^\infty$ is also a basis for $\mathcal{X}$. 
\end{theorem}

Since then, the Paley-Wiener perturbation theorem has been generalized to many related topics, including
entire functions of exponential type \cite{young2001introduction} and frames in Hilbert spaces \cite{christensen1995paley, christensen1995frame, cazassa1997perturbation} and their extensions, such as Banach frames  \cite{christensen1997perturbations}, g-frames \cite{sun2007stability}, and operator represented frames \cite{christensen2017operator}. As a generalization of orthonormal bases, frames were first introduced by Duffin and Schaeffer in the context of nonharmonic Fourier analysis \cite{duffin1952class} and have been applied in many areas, such as the Kadison-Singer problem \cite{casazza2013kadison},  time-frequency analysis\cite{grochenig2001foundations}, and wavelet analysis\cite{daubechies1992ten}.   

Recall that a sequence $\{{\bf f}_i\}_{i=1}^\infty$ in a separable Hilbert space $\mathcal{H}$ is called a \textit{frame} for $\mathcal{H}$ if  there exist $0<A \leq B < \infty$ (frame bounds) such that for any ${\bf f} \in \mathcal{H}$, 
\begin{equation*}
    A \Vert {\bf f}  \Vert^2 \leq\sum_{i =1}^{\infty} \vert \langle {\bf f},{\bf f}_i \rangle \vert ^2  \leq B\Vert {\bf f} \Vert^2,
\end{equation*} 
and $\{{\bf f}_i\}_{i=1}^\infty$ is called a tight frame if $A = B$ and Parseval if $A = B=1$. 
Compared to a basis in $\mathcal{H}$, a frame permits linear dependence between frame elements and allows each vector in $\mathcal{H}$ to be written as a linear combination of frame elements in a redundant way. Given a frame $\{{\bf f}_i\}_{i=1}^\infty$ in $\mathcal{H}$, the associated \textit{frame operator} ${\bf S}: \mathcal{H} \rightarrow \mathcal{H}$ is given by 
\begin{equation*}
     {\bf S}({\bf f}) = \sum_{i =1}^{\infty}  \langle {\bf f},{\bf f}_i \rangle  {\bf f}_i, \ \text{for any} \ {\bf f} \in \mathcal{H}.
\end{equation*}
It is well-known that ${\bf S}: \mathcal{H} \rightarrow \mathcal{H}$ is bounded, positive, self-adjoint, and invertible with a bounded inverse. In signal processing applications, a signal vector from $\mathcal{H}$ can be reconstructed using tight frames and dual frames. If $\{{\bf f}_i\}_{i=1}^\infty$ is a tight frame for $\mathcal{H}$ with bound $A>0$, then $\mathbf{S} =A \mathbf{Id}$ where $\mathbf{Id}$ is the identity operator on $\mathcal{H}$.   Thus for any ${\bf f} \in \mathcal{H}$, 
$${\bf f} = \frac{1}{A} \sum\limits_{i =1}^{\infty} \langle {\bf f},{\bf f}_i \rangle  {\bf f}_i.$$
The second way to reconstruct ${\bf f}$ is to use dual frames. A sequence $\{{\bf g}_i\}_{i=1}^\infty$ in $\mathcal{H}$ is called a \textit{dual frame} of the frame $\{{\bf f}_i\}_{i=1}^\infty$ if for any ${\bf f} \in \mathcal{H}$, 
\begin{equation*}
   {\bf f} = \sum_{i =1}^{\infty}  \langle {\bf f},{\bf g}_i \rangle  {\bf f}_i   = \sum_{i =1}^{\infty}  \langle {\bf f},{\bf f}_i \rangle  {\bf g}_i.
\end{equation*}
An example of dual frames is the \emph{canonical dual} $\{{\bf S}^{-1}{\bf f}_i\}_{i=1}^\infty$, where ${\bf S}$ is the frame operator of $\{{\bf f}_i\}_{i=1}^\infty$. When $\mathcal{H}$ is finite-dimensional, we get finite frames and the sums in the above definitions become finite. See \cite{christensen2016introduction} for more details on frames.

As a particular case of the Paley-Wiener perturbation theorem for frames, it is well-known (see, e.g.,  \cite{christensen1995frame}) that if a sequence $\{{\bf h}_i\}_{i=1}^\infty \subset \mathcal{H}$ is quadratically close to a frame $\{{\bf f}_i\}_{i=1}^\infty$ with bounds $0<A \leq B$, that is, if
$$K:=\sum\limits_{i=1}^\infty  \Vert  {\bf f}_i-{\bf h}_i  \Vert^2 < A,$$
then $\{{\bf h}_i\}_{i=1}^\infty$ is a frame for $\mathcal{H}$ with bounds $(\sqrt{A}-\sqrt{K})^2$ and $(\sqrt{B}+\sqrt{K})^2$. 

In this work, we study perturbations of probabilistic frames. As introduced in \cite{ehler2012random}, developed in \cite{ehler2012minimization}, and further reviewed in \cite{ehler2013probabilistic}, a probability measure $\mu$ on $\mathbb{R}^n$ is called a \textit{probabilistic frame} if there exist $0<A \leq B$ such that for any ${\bf x} \in \mathbb{R}^n$, 
\begin{equation*}
    A \Vert {\bf x }  \Vert^2 \leq \int_{\mathbb{R}^n} \vert \langle {\bf x},{\bf y} \rangle \vert ^2 d\mu({\bf y})  \leq B\Vert {\bf x}  \Vert^2.
\end{equation*}
It has been shown that the set of probabilistic frames is open in the $2$-Wasserstein topology \cite{chen2025probabilistic}, meaning that if a probability measure is close enough to a probabilistic frame in the $2$-Wasserstein metric, this measure is also a probabilistic frame. 

One of the main contributions in this work is that we also find explicit frame bounds for such probability measures in Corollary \ref{openness}, which reestablishes the openness of the set of probabilistic frames. The following proposition is a quadratic-closeness perturbation of probabilistic frames, where $\mathcal{P}_2(\mathbb{R}^n)$ is the set of probability measures on $\mathbb{R}^n$ with finite second moments and $M_2(\nu)$ is the second moment of $\nu$.   
\begin{proposition} \label{QuadClose}
Let $\mu$ be a probabilistic frame with lower bound  $A>0$ and let $\nu \in \mathcal{P}_2(\mathbb{R}^n)$. Suppose there exists a transport coupling $\gamma \in \Gamma(\mu, \nu)$ such that
\begin{equation*}
   \lambda:= \int_{\mathbb{R}^n  \times \mathbb{R}^n} \Vert {\bf x}-{\bf y} \Vert^2  d \gamma({\bf x}, {\bf y}) < A,
\end{equation*}
then $\nu$ is a probabilistic frame with bounds $(\sqrt{A}-\sqrt{\lambda})^2 $ and $M_2(\nu)$. 
\end{proposition}

The following corollary is clear if $\gamma \in \Gamma(\mu, \nu)$ in Proposition \ref{QuadClose} is optimal for the $2$-Wasserstein distance $W_2(\mu, \nu)$.

\begin{corollary}\label{openness}
    Let $\mu$ be a probabilistic frame with lower bound  $A>0$ and let $\nu \in \mathcal{P}_2(\mathbb{R}^n)$. Suppose $W_2(\mu, \nu) < \sqrt{A}$,
then $\nu$ is a probabilistic frame with bounds $(\sqrt{A}-W_2(\mu, \nu))^2 $ and $M_2(\nu)$. Furthermore, the set of probabilistic frames forms an open subset of $\mathcal{P}_2(\mathbb{R}^n)$ endowed with the  $2$-Wasserstein topology.
\end{corollary}

This paper is organized as follows. In Section \ref{preliminaries}, we introduce probabilistic frames, optimal transport, and the invertibility of linear operators on Banach spaces.  
In Section \ref{FramePerturbation}, we generalize the Paley-Wiener perturbation theorem to probabilistic frames. In particular, we give sufficient perturbation conditions using probabilistic dual frames. Finally, in Section \ref{sec:proof}, we prove Proposition \ref{QuadClose}.

\section{Mathematical Preliminaries}\label{preliminaries}

This section begins by recalling probabilistic frames and optimal transport. Throughout the paper, we use ${\bf A}$ and ${\bf x}$ to denote a real matrix and a vector in the Euclidean space $\mathbb{R}^n$, and $A$ and $x$ to represent real numbers. In particular, ${\bf 0}$ is used to denote the zero vector in $\mathbb{R}^n$, ${\bf 0}_{n \times n}$ the $n \times n$ zero matrix, and $0$ the number zero. We use ${\bf Id}$ for the $n \times n$ identity matrix, and  ${\bf A}^t$ for the transpose of the matrix ${\bf A}$. 

\subsection{Probabilistic Frames and Optimal Transport}
Let $\mathcal{P}(\mathbb{R}^n)$ be the set of Borel probability measures on $\mathbb{R}^n$. For $\mu \in \mathcal{P}(\mathbb{R}^n)$, we define its second moment as 
$$M_2(\mu) := \int_{\mathbb{R}^n} \Vert {\bf x} \Vert ^2 d\mu({\bf x}).$$
Let
$\mathcal{P}_2(\mathbb{R}^n) \subset \mathcal{P}(\mathbb{R}^n)$ be the set of probability measures on $\mathbb{R}^n$ with finite second moments, that is, 
\begin{equation*}
\mathcal{P}_2(\mathbb{R}^n) := \left \{ \mu \in \mathcal{P}(\mathbb{R}^n):  \int_{\mathbb{R}^n} \Vert {\bf x} \Vert ^2 d\mu({\bf x}) < + \infty \right \}.
\end{equation*}
In addition, we define the \textit{support} of $\mu \in  \mathcal{P}(\mathbb{R}^n)$ as 
$$\supp(\mu) := \{ {\bf x} \in \mathbb{R}^n: \mu(B_r({\bf x}))>0 \ \text{for all} \ r>0 \},$$
where $B_r({\bf x})$ is the open ball centered at ${\bf x}$ with radius $r>0$.
If $\mu \in  \mathcal{P}(\mathbb{R}^n) $ and $f: \mathbb{R}^n \rightarrow \mathbb{R}^m$ is a Borel measurable map where $n$ may differ from $m$, then $f_{\#} \mu \in \mathcal{P}(\mathbb{R}^m)$ is called the \textit{pushforward} of $\mu$ by the map $f$,  and is defined as
$$f_{\#} \mu (E) := \mu \big (f^{-1}(E) \big ), \ \text{for any Borel set}  \ E \subset \mathbb{R}^m.$$ If $f$ is linear and represented by a matrix ${\bf A}$ with respect to some basis, then ${\bf A}_{\#} \mu$ is used to denote $f_{\#} \mu$. In particular, $(\mathbf{Id}, f)$ is used to denote a map from $\mathbb{R}^n$ to $\mathbb{R}^n \times \mathbb{R}^m$ via $\mathbf{x} \mapsto (\mathbf{x}, f(\mathbf{x}))$. In this case, $(\mathbf{Id}, f)_\#\mu$ is a probability measure on $\mathbb{R}^n \times \mathbb{R}^m$ that is supported on the graph of $f$.

Note that if $\{{\bf y}_i\}_{i=1}^N$ is a frame for $\mathbb{R}^n$ with bounds $0<A \leq B$ where $N \geq n$, then by letting $\mu_f := \frac{1}{N} \sum\limits_{i=1}^N  \delta_{{\bf y}_i} \in \mathcal{P}(\mathbb{R}^n)$, the frame inequality of $\{{\bf y}_i\}_{i=1}^N$ becomes
\begin{equation*}
    \frac{A}{N} \Vert {\bf x}  \Vert^2 \leq\int_{\mathbb{R}^n} \vert \langle {\bf x},{\bf y} \rangle \vert ^2 d\mu_f({\bf y})  \leq \frac{B}{N} \Vert {\bf x}  \Vert^2, \ \text{for any} \ {\bf x} \in \mathbb{R}^n.
\end{equation*}
This means that finite frames can be viewed as discrete probabilistic measures. 
Inspired by this observation, Martin Ehler \cite{ehler2012random} introduced the concept of probabilistic frames.
Subsequently, Martin Ehler and Kasso A. Okoudjou investigated probabilistic frame potential \cite{ehler2012minimization} and later surveyed this area,  with particular emphasis on its connections to optimal transport \cite{ehler2013probabilistic}.

\begin{definition}
$\mu \in  \mathcal{P}(\mathbb{R}^n)$ is called a \textit{probabilistic frame} for $\mathbb{R}^n$ if there exist $0<A \leq B < \infty$ (frame bounds) such that for any ${\bf x} \in \mathbb{R}^n$, 
\begin{equation*}
    A \Vert {\bf x}  \Vert^2 \leq\int_{\mathbb{R}^n} \vert \langle {\bf x},{\bf y} \rangle \vert ^2 d\mu({\bf y})  \leq B \Vert {\bf x}  \Vert^2.
\end{equation*}
$\mu$ is called a tight probabilistic frame if $A = B$ and Parseval if $A = B=1$. Moreover, $\mu$ is said to be a Bessel probability measure if only the upper bound holds. 
\end{definition}

If $\mu \in  \mathcal{P}_2(\mathbb{R}^n)$, then the Cauchy-Schwarz inequality implies that $\mu$ is Bessel with bound $M_2(\mu) $. Furthermore, one can define the \textit{frame operator} ${\bf S}_\mu$ for $\mu$  as 
    \begin{equation*}
          {\bf S}_\mu := \int_{\mathbb{R}^n}  {\bf y} {\bf y}^t d \mu({\bf y}).
    \end{equation*}
It has been shown that frame operators can characterize probabilistic frames.

 \begin{proposition}[Theorem 12.1 in \cite{ehler2013probabilistic}, Proposition 3.1 in \cite{maslouhi2019probabilistic}] \label{TAcharacterization}
Let $\mu  \in \mathcal{P}(\mathbb{R}^n)$. 
\begin{itemize}
    \item[$(1)$] $\mu$ is a probabilistic frame  $\Leftrightarrow$ $\mu \in \mathcal{P}_2(\mathbb{R}^n)$ and $\spanning\{\supp(\mu)\} = \mathbb{R}^n$.
    
     \item[$(2)$] If $\mu  \in \mathcal{P}_2(\mathbb{R}^n)$, then $\mu$ is a probabilistic frame   $\Leftrightarrow$ ${\bf S}_{\mu}$ is positive definite; and $\mu$ is a tight probabilistic frame with bound $A > 0$ $\Leftrightarrow$ ${\bf S}_{\mu} = A {\bf Id}$. 
\end{itemize}
\end{proposition}

Since probabilistic frames lie in $\mathcal{P}_2(\mathbb{R}^n)$, they can be studied using optimal transport and the $2$-Wasserstein distance \cite{ehler2013probabilistic, wickman2014optimal}. Given two probabilistic frames $\mu$ and $\nu$, let $\Gamma(\mu,\nu)$ be the set of transport couplings with marginals $\mu$ and $\nu$, that is,
$$ \Gamma(\mu,\nu) :=  \left \{ \gamma \in \mathcal{P}(\mathbb{R}^n \times \mathbb{R}^n): {\pi_{{ x}}}_{\#} \gamma = \mu, \ {\pi_{{ y}}}_{\#} \gamma = \nu \right \},$$
where $\pi_{{ x}}$, $\pi_{{ y}}$ are projections onto the ${\bf x}$ and ${\bf y}$ coordinates: for any $({\bf x},{\bf y}) \in \mathbb{R}^n \times \mathbb{R}^n $, $\pi_{{ x}}({\bf x}, {\bf y}) = {\bf x}$ and $ \pi_{{ y}}({\bf x},{\bf y}) = {\bf y}$. The $2$-Wasserstein distance $ W_2(\mu,\nu)$ is often used to quantify the distance between probabilistic frames $\mu$ and $\nu$:  
$$ W_2(\mu,\nu) := \left (\underset{\gamma \in \Gamma(\mu,\nu)}{\inf}  \int_{\mathbb{R}^n \times \mathbb{R}^n}  \left \Vert {\bf x}-{\bf y}\right \Vert^2 \ d\gamma({\bf x},{\bf y}) \right)^{1/2}.$$
This is of interest even if we are only concerned with finite frames, since the $2$-Wasserstein distance can quantify the distance between probabilistic frames induced by finite frames of different cardinalities. If $\gamma^* \in \Gamma(\mu, \nu)$ is an optimizer for $ W_2(\mu,\nu)$, then $\gamma^*$ is called an optimal transport coupling for $ W_2(\mu,\nu)$.

Similar to frames, one can reconstruct a signal using tight probabilistic frames and \textit{probabilistic dual frames}. It has been shown that every probabilistic dual frame (also known as transport dual) is a probabilistic frame \cite{wickman2014optimal, wickman2017duality}.

\begin{definition}
Let $\mu$ be a probabilistic frame for $\mathbb{R}^n$. Then $\nu \in \mathcal{P}_2(\mathbb{R}^n)$ is called a \textit{probabilistic dual frame} of $\mu$ if there exists $\gamma \in \Gamma(\mu, \nu)$ such that
\begin{equation*}
  \int_{\mathbb{R}^n \times \mathbb{R}^n} {\bf x}{\bf y}^t d\gamma({\bf x}, {\bf y}) = {\bf Id}.
\end{equation*}
\end{definition}

If $\mu$ is a probabilistic frame with bounds $0<A \leq B$, then ${{\bf S}_\mu^{-1}}_{\#}\mu$ is called the \emph{canonical probabilistic dual frame} of $\mu$ (with bounds $\frac{1}{B}$ and $ \frac{1}{A}$), because one can let $\gamma$ be ${({\bf Id}, {\bf S}_\mu^{-1})}_{\#}\mu \in \Gamma(\mu, {{\bf S}_\mu^{-1}}_{\#}\mu)$ in the above definition of probabilistic dual frame. In this case, we have the following  reconstruction: for any ${\bf f} \in \mathbb{R}^n$, 
\begin{equation}\label{reconstruction2}
   {\bf f}=   \int_{\mathbb{R}^n} \langle  {\bf f}, {\bf S}_\mu^{-1} {\bf x} \rangle {\bf x}  d\mu({\bf x}) =\int_{\mathbb{R}^n} \langle  {\bf S}_\mu^{-1} {\bf f}, {\bf x} \rangle {\bf x}  d\mu({\bf x}). 
\end{equation}
Another special frame related to $\mu$ is the \emph{canonical Parseval probabilistic frame} ${{\bf S}_{\mu}^{-1/2}}_{\#}\mu$, and the associated reconstruction is given as follows: for any ${\bf f} \in \mathbb{R}^n$, 
\begin{equation}\label{reconstruction1}
   {\bf f}=\int_{\mathbb{R}^n} \langle {\bf f}, {\bf S}_\mu^{-1/2} {\bf x} \rangle {\bf S}_\mu^{-1/2} {\bf x}   d\mu({\bf x})= \int_{\mathbb{R}^n} \langle {\bf S}_\mu^{-1/2} {\bf f}, {\bf x} \rangle {\bf S}_\mu^{-1/2} {\bf x}   d\mu({\bf x}).
\end{equation}

For interested readers, see \cite{ehler2012random, ehler2011frame, chen2025probabilistic, maslouhi2019probabilistic, ehler2013probabilistic, wickman2017duality, ehler2012minimization, wickman2014optimal,cheng2019optimal, loukili2020minimization, wickman2023gradient} for more details on probabilistic frames and \cite{figalli2021invitation} for optimal transport.

We conclude this subsection with basic properties of the frame operator. If $\mu$ is a probabilistic frame with bounds $0< A \leq B$, then the associated frame inequality shows that every eigenvalue of ${\bf S}_\mu$ is between $A$ and $B$.  Now let $\Vert {\bf S}_\mu \Vert_2$ be the 2-matrix norm of ${\bf S}_\mu$, which equals the largest eigenvalue of ${\bf S}_{\mu}$. Then
\begin{equation*}
   A \leq  \Vert {\bf S}_\mu \Vert_2 \leq B, \  \frac{1}{B} \leq  \Vert {\bf S}_\mu^{-1} \Vert_2 \leq \frac{1}{A},  \ \text{and} \  \frac{1}{\sqrt{B}} \leq  \Vert {\bf S}_\mu^{-1/2} \Vert_2 \leq \frac{1}{\sqrt{A}}.
\end{equation*}

\subsection{Invertibility of Linear Operators on Banach Spaces} This subsection introduces the invertibility of linear operators on Banach spaces.  It is well-known that a bounded linear operator $U$ on a Banach space $\mathcal{X}$ is invertible and $ \Vert U^{-1} \Vert \leq \frac{1}{1-\Vert I-U \Vert }$ if $\Vert I-U \Vert <1$,
where $I$  is the identity operator on $\mathcal{X}$. Peter G. Casazza and Ole Christensen generalized this result to the following broader setting.
\begin{lemma}[Lemma 1 in \cite{cazassa1997perturbation}]\label{InverseType2}
    Let $\mathcal{X}$ be a Banach space and  $U: \mathcal{X} \rightarrow \mathcal{X} $ be a linear operator on $\mathcal{X}$. If there exist  $\lambda_1, \lambda_2 \in [0,1)$ such that for any ${\bf x} \in \mathcal{X}$, 
    \begin{equation*}
        \Vert U{\bf x} -{\bf x} \Vert \leq \lambda_1 \Vert {\bf x} \Vert + \lambda_2 \Vert U{\bf x} \Vert,
    \end{equation*}
then $U$ is bounded and invertible, and for any ${\bf x} \in \mathcal{X}$, 
\begin{equation*}
    \frac{1-\lambda_1}{1+\lambda_2} \Vert {\bf x} \Vert \leq \Vert U{\bf x} \Vert  \leq \frac{1+\lambda_1}{1-\lambda_2} \Vert {\bf x} \Vert \ \text{and} \ \frac{1-\lambda_2}{1+\lambda_1} \Vert {\bf x} \Vert \leq \Vert U^{-1}{\bf x} \Vert  \leq \frac{1+\lambda_2}{1-\lambda_1} \Vert {\bf x} \Vert.
\end{equation*}
\end{lemma}

The following corollary concerns the extension of a linear operator on a dense subspace to a bounded linear operator on the whole space.

\begin{corollary}[Remark following Corollary 1 in \cite{cazassa1997perturbation}]\label{extension}
   Suppose $\mathcal{X}$ and $\mathcal{Y}$ are Banach spaces. Let $U:\mathcal{X} \rightarrow \mathcal{Y}$ be a bounded linear operator, $\mathcal{X}_0$ a dense subspace of $\mathcal{X}$, and $T:\mathcal{X}_0 \rightarrow \mathcal{Y}$ a linear mapping. If there exist $\delta, \lambda_1 \geq 0 $ and $\lambda_2 \in [0,1)$ such that
    \begin{equation*}
        \Vert U{\bf x} -T{\bf x} \Vert \leq \lambda_1 \Vert U{\bf x} \Vert + \lambda_2 \Vert T{\bf x} \Vert + \delta \Vert {\bf x} \Vert, \ \text{for any ${\bf x} \in \mathcal{X}_0$,}
    \end{equation*}
then $T$ is bounded on the dense subspace $\mathcal{X}_0$ and can be uniquely extended to a bounded linear operator $\Tilde{T}: \mathcal{X} \rightarrow \mathcal{Y}$, and the inequality still holds for $\Tilde{T}$ and ${\bf x} \in \mathcal{X}$.  
\end{corollary}

\begin{proof}
    For any ${\bf x} \in \mathcal{X}_0$, we have
     \begin{equation*}
      \Vert T{\bf x} \Vert \leq  \Vert U{\bf x} \Vert +  \Vert U{\bf x} -T{\bf x} \Vert \leq (1+\lambda_1) \Vert U{\bf x} \Vert + \lambda_2 \Vert T{\bf x} \Vert + \delta \Vert {\bf x} \Vert.
    \end{equation*}
Therefore, for any ${\bf x} \in \mathcal{X}_0$, 
    \begin{equation*}
      \Vert T{\bf x} \Vert  \leq \frac{1+\lambda_1}{1-\lambda_2} \Vert U{\bf x} \Vert + \frac{\delta}{1-\lambda_2}  \Vert {\bf x} \Vert,
    \end{equation*}
which means that $T$ is bounded on the dense subspace $\mathcal{X}_0$. By standard analysis, $T$ can be uniquely extended to a bounded linear operator $\Tilde{T}: \mathcal{X} \rightarrow \mathcal{Y}$ defined by
$$\Tilde{T}{\bf x} = \lim\limits_{k \rightarrow \infty} T{\bf x}_k,$$
where ${\bf x} \in \mathcal{X}$, $\{{\bf x}_k\} \subset \mathcal{X}_0$, and ${\bf x}_k \rightarrow {\bf x}$. The definition of $\Tilde{T}{\bf x}$ is also independent of the approximating sequence: if ${\bf x}_k\to{\bf x}$ and ${\bf y}_k\to{\bf x}$, then $\|T{\bf x}_k-T{\bf y}_k\| \leq \|T\|\ \|{\bf x}_k-{\bf y}_k\|\to 0$, where $\|T\|$ is the norm of $T$ on $\mathcal{X}_0$.
Then for any ${\bf x} \in \mathcal{X}$, 
\begin{equation*}
\begin{split}
     \Vert U{\bf x} -\Tilde{T}{\bf x} \Vert= \lim_{k \rightarrow \infty} \| U{\bf x}_k -T{\bf x}_k \|
      &\leq  \lim_{k \rightarrow \infty} (\lambda_1 \Vert U{\bf x}_k \Vert + \lambda_2 \Vert T{\bf x}_k \Vert + \delta \Vert {\bf x}_k \Vert) \\
      & = \lambda_1 \Vert U{\bf x} \Vert + \lambda_2 \Vert \Tilde{T}{\bf x} \Vert + \delta \Vert {\bf x} \Vert.
\end{split}
    \end{equation*}
\end{proof}

\section{Probabilistic Frame Perturbations}\label{FramePerturbation}
In this section, we first list the associated results for frame perturbations, and then we obtain the counterparts for probabilistic frames. The whole section is divided into two parts. The first part concerns the extension of the Paley-Wiener perturbation theorem to probabilistic frames. The second part gives sufficient perturbation conditions where probabilistic dual frames are used.  

Peter G. Casazza and Ole Christensen generalized the Paley-Wiener perturbation theorem to study the stability of frames in Hilbert spaces as follows. Note that although the sequences $\{{\bf f}_i\}_{i=1}^{\infty}$ and $\{{\bf g}_i\}_{i=1}^{\infty}$
are infinite, the perturbation estimates only require finite weighted sums with arbitrary scalars $\{c_i\}_{i=1}^n$ where $n \in \mathbb{N}$.

\begin{theorem}[Theorem 2 in \cite{cazassa1997perturbation}]\label{PaleyFrame2}
Let $\{{\bf f}_i\}_{i=1}^{\infty}$ be a frame for the Hilbert space $\mathcal{H}$ with bounds $0<A \leq B$ and let $\{{\bf g}_i\}_{i=1}^{\infty}$ be a sequence in $\mathcal{H}$. Suppose there exist constants $\lambda_1,\lambda_2, \delta \geq 0$ such that $\max (\lambda_1 + \frac{\delta}{\sqrt{A}}, \lambda_2) < 1$ and for all scalars $c_1, \dots, c_n \ (n=1, 2, \dots)$,
    \begin{equation*}
  \Big   \Vert \sum\limits_{i=1}^n c_i ({\bf f}_i-{\bf g}_i) \Big \Vert \leq \lambda_1  \Big \Vert \sum\limits_{i=1}^n c_i {\bf f}_i \Big \Vert + \lambda_2  \Big \Vert \sum\limits_{i=1}^n c_i {\bf g}_i \Big \Vert + \delta  \Big [ \sum\limits_{i=1}^n |c_i|^2 \Big ]^{1/2},
    \end{equation*}
then $\{{\bf g}_i\}_{i=1}^{\infty}$ is a frame for $\mathcal{H}$ with bounds
\begin{equation*}
   A \Big (1-\frac{\lambda_1+\lambda_2 + \frac{\delta}{\sqrt{A}}}{1+\lambda_2} \Big )^2 \ \text{and} \ B \Big (1+\frac{\lambda_1+\lambda_2 + \frac{\delta}{\sqrt{B}}}{1-\lambda_2} \Big )^2.
\end{equation*}
\end{theorem}

Furthermore, one can use dual frames in the frame perturbation.
\begin{theorem}[Theorem 2.1 in \cite{chen2014perturbations}]\label{PaleyDualFrames}
    Let $\{{\bf f}_i\}_{i=1}^{\infty}$ be a frame for the Hilbert space $\mathcal{H}$ with bounds $0<A \leq B$ and let $\{{\bf h}_i\}_{i=1}^{\infty}$ be a dual frame of $\{{\bf f}_i\}_{i=1}^{\infty}$ with upper frame bound $D>0$. Suppose $\{{\bf g}_i\}_{i=1}^{\infty}$ is a sequence in $\mathcal{H}$ such that 
    \begin{equation*}
       \alpha:= \sum_{i=1}^\infty \Vert {\bf f}_i-{\bf g}_i \Vert^2 < +\infty  \ \text{and} \ \beta:= \sum_{i=1}^\infty \Vert {\bf f}_i-{\bf g}_i \Vert \Vert {\bf h}_i \Vert < 1,
    \end{equation*}
then $\{{\bf g}_i\}_{i=1}^{\infty}$ is a frame in $\mathcal{H}$ with bounds $\frac{(1-\beta)^2}{D}$ and $B(1+\sqrt{\frac{\alpha}{B}})^2$.
\end{theorem}

\subsection{Paley-Wiener Theorem for Probabilistic Frames}\label{PaleyWiener}
In this subsection, we establish the Paley-Wiener perturbation theorem for probabilistic frames, where Lemma \ref{InverseType2} and Corollary \ref{extension} on the invertibility and extension of linear operators are utilized. We denote by $C_c(\mathbb{R}^n)$ the set of continuous functions on $\mathbb{R}^n$ with compact support, which serves as the test function space. Note that $C_c(\mathbb{R}^n)$ is dense in $L^2(\mu)$ when $\mu$ is a Borel probability measure on $\mathbb{R}^n$.

\begin{theorem}\label{PaleyPFrame2}
Let $\mu$ be a probabilistic frame for $\mathbb{R}^n$ with bounds $0<A\leq B$ and $\nu \in \mathcal{P}_2(\mathbb{R}^n)$. If there exist  $\lambda_1,\lambda_2, \delta \geq 0$ such that $\mathrm{max} (\lambda_1 + \frac{\delta}{\sqrt{A}}, \lambda_2) <1$ and for all $w \in C_c(\mathbb{R}^n)$,
    \begin{equation*}
    \begin{split}
        \Big   \Vert \int_{\mathbb{R}^n} w({\bf x}){\bf x} d\mu({\bf x}) & - \int_{\mathbb{R}^n} w({\bf y}){\bf y} d\nu({\bf y}) \Big \Vert  \\
       &\leq \lambda_1 \Big \Vert \int_{\mathbb{R}^n} w({\bf x}){\bf x} d\mu({\bf x}) \Big \Vert + \lambda_2 \Big \Vert \int_{\mathbb{R}^n} w({\bf y}){\bf y} d\nu({\bf y}) \Big \Vert + \delta  \Vert  w \Vert_{L^2(\mu)}, 
    \end{split}
    \end{equation*}
then $\nu$ is a probabilistic frame for $\mathbb{R}^n$ with bounds 
$\frac{ A^2(1-(\lambda_1 + \frac{\delta}{\sqrt{A}}))^2}{(1+\lambda_2)^2 M_2(\nu)} \ \text{and} \ M_2(\nu).$
\end{theorem}
\begin{proof}
Since $\nu \in \mathcal{P}_2(\mathbb{R}^n)$, $\nu$ is Bessel  with bound $M_2(\nu)$. Now let us get the lower frame bound. Define the synthesis operator $U:L^2(\mu) \rightarrow \mathbb{R}^n$ for $\mu$ as 
\begin{equation*}
    U(w) := \int_{\mathbb{R}^n} w({\bf x}) {\bf x} d\mu({\bf x}).
    \end{equation*}
Clearly, $U$ is bounded and $\Vert U \Vert^2 \leq M_2(\mu)$. Then its adjoint (analysis) operator $U^*: \mathbb{R}^n \rightarrow L^2(\mu)$ exists and is given by $U^*({\bf x}) = \langle {\bf x}, \cdot \rangle$. Note that $UU^* = {\bf S}_{\mu}$ where ${\bf S}_{\mu}$ is the frame operator of $\mu$. By Proposition \ref{TAcharacterization}, ${\bf S}_{\mu}$ is positive definite. 
So $UU^*$ is invertible and $(UU^*)^{-1} = {\bf S}_{\mu}^{-1}$. Define  a linear operator $T: C_c(\mathbb{R}^n) \rightarrow \mathbb{R}^n$ as
\begin{equation*}
    T(w) := \int_{\mathbb{R}^n} w({\bf y}) {\bf y} d\nu({\bf y}).
\end{equation*}
By the theorem assumption, there exist  $\lambda_1,\lambda_2, \delta \geq 0$ such that $\mathrm{max} (\lambda_1 + \frac{\delta}{\sqrt{A}}, \lambda_2) <1$ and for all $w \in C_c(\mathbb{R}^n)$,
\begin{equation*}
    \Vert U(w) - T(w) \Vert \leq \lambda_1 \Vert U(w)  \Vert + \lambda_2  \Vert T(w) \Vert + \delta  \Vert  w \Vert_{L^2(\mu)} .
    \end{equation*}
Since $C_c(\mathbb{R}^n)$ is dense in $L^2(\mu)$, by Corollary \ref{extension}, $T$ is bounded on $C_c(\mathbb{R}^n)$ and can be uniquely extended to a bounded linear operator $\Tilde{T}: L^2(\mu) \rightarrow \mathbb{R}^n$ given by 
$$
 \Tilde{T}(w) := \lim_{k \rightarrow \infty} T(w_k) = \lim_{k \rightarrow \infty} \int_{\mathbb{R}^n} w_k({\bf y}) {\bf y} d\nu({\bf y}),
$$
where $w \in L^2(\mu)$, $\{w_k\} \subset C_c(\mathbb{R}^n)$, and $w_k$ converges to $w$ in $L^2(\mu)$. Furthermore, for any $w \in L^2(\mu)$, we still have
\begin{equation}\label{PaleyWienerType2}
    \Vert U(w) - \Tilde{T}(w) \Vert \leq \lambda_1 \Vert U(w)  \Vert + \lambda_2  \Vert \Tilde{T}(w) \Vert + \delta  \Vert  w \Vert_{L^2(\mu)} .
    \end{equation}
Therefore, for any $w \in L^2(\mu)$,
\begin{equation*}
       \Vert \Tilde{T}(w)  \Vert \leq   \Vert U(w)  \Vert +    \Vert U(w) - \Tilde{T}(w)  \Vert \leq \big((\lambda_1 +1)  \Vert U  \Vert +\delta  \big) \Vert  w \Vert_{L^2(\mu)} + \lambda_2\Vert \Tilde{T}(w)  \Vert.
\end{equation*}
Thus 
$$\Vert \Tilde{T}  \Vert \leq \frac{ (\lambda_1 +1)  \Vert U  \Vert +\delta}{1-\lambda_2} < +\infty.$$
Now define $U^{+}: \mathbb{R}^n \rightarrow L^2(\mu)$ by 
\begin{equation*}
    (U^{+}{\bf x})(\cdot) := (U^*(UU^*)^{-1}{\bf x}) (\cdot) = (U^*({\bf S}_{\mu}^{-1}{\bf x})) (\cdot) = \langle {\bf S}_{\mu}^{-1}{\bf x}, \cdot \rangle   \in L^2(\mu),
\end{equation*}
where $U^*$ is the adjoint operator of $U$. Then 
\begin{equation*}
\begin{split}
        \Vert U^{+}{\bf x} \Vert_{L^2(\mu)}^2 &= \int_{\mathbb{R}^n} \left|\langle {\bf S}_{\mu}^{-1}{\bf x}, {\bf y} \rangle  \right|^2 d\mu({\bf y}) 
        =  \int_{\mathbb{R}^n} \left|\langle {\bf x},{\bf y} \rangle \right|^2 d  ({{\bf S}_{\mu}^{-1}}_{\#}\mu)({\bf y}) \leq \frac{1}{A} \Vert  {\bf x} \Vert^2,
\end{split}
\end{equation*}
where the inequality follows from that ${{\bf S}_{\mu}^{-1}}_{\#}\mu$ is a probabilistic frame with upper bound $\frac{1}{A}$. Then, for any ${\bf x} \in \mathbb{R}^n$,  replacing $w$ in Equation \eqref{PaleyWienerType2} with $U^{+}{\bf x}$ leads to 
\begin{equation*}
 \begin{split}
       \Vert {\bf x} - \Tilde{T}(U^{+}{\bf x}) \Vert   \leq (\lambda_1 + \frac{\delta}{\sqrt{A}}) \Vert  {\bf x} \Vert + \lambda_2 \Vert \Tilde{T}(U^{+}{\bf x})  \Vert.
 \end{split}
\end{equation*}
Since $ \max (\lambda_1 + \frac{\delta}{\sqrt{A}}, \lambda_2) <1$,  Lemma \ref{InverseType2} implies that $\Tilde{T}U^{+}: \mathbb{R}^n \rightarrow \mathbb{R}^n$ is invertible and $\Vert (\Tilde{T}U^{+})^{-1}  \Vert \leq \frac{1+\lambda_2}{1- (\lambda_1 + \frac{\delta}{\sqrt{A}})}.$
Therefore, each ${\bf x} \in \mathbb{R}^n$ could be written as 
\begin{equation}\label{eq:reconstruction}
    {\bf x} = \Tilde{T}U^{+}(\Tilde{T}U^{+})^{-1}{\bf x} = \Tilde{T}\left(\langle {\bf S}_{\mu}^{-1}(\Tilde{T}U^{+})^{-1}{\bf x}, \cdot \rangle \right).
\end{equation}
Note that for the map $w({\bf y})=\langle {\bf a},{\bf y}\rangle$ where ${\bf a} \in \mathbb R^n$ is given, there exist $\{w_k\} \in C_c(\mathbb R^n)$ such that $w_k$ not only converges to $w$ in $L^2(\mu)$, but also $w_k$ converges to $w$ pointwise. Indeed, define $\psi:[0,\infty) \to [0,1]$ by $\psi(t) = 
\begin{cases}
1, & 0  \leq t \leq 1 \\
2-t, & 1 < t < 2 \\
0, & t \geq 2
\end{cases}.$
Clearly, $\psi$ is continuous and has compact support. 
For each $k\geq 1$, define
\[
\phi_k({\bf y}):=\psi\left(\frac{\|{\bf y}\|}{k}\right) \ \text{and} \
w_k({\bf y}):=\phi_k({\bf y}) w({\bf y}).
\]
Then $w_k\in C_c(\mathbb R^n)$, $0 \leq \phi_k \leq 1$, and
$\phi_k({\bf y}) \to 1$ for each ${\bf y}$. So $w_k $ converges to $w$ pointwise.
Moreover, since $\mu \in \mathcal P_2(\mathbb R^n)$, then
\[
\|w_k-w\|_{L^2(\mu)}^2
=\int_{\mathbb R^n} |\langle {\bf a},{\bf y}\rangle|^2 |1-\phi_k({\bf y})|^2\ d\mu({\bf y})
\leq \|{\bf a}\|^2 \int_{\{\|{\bf y}\|>k\}} \|{\bf y}\|^2\,d\mu({\bf y})\to 0,
\]
which shows that $w_k$ converges to $w$ in $L^2(\mu)$. Thus, for the map $\langle {\bf S}_{\mu}^{-1}(\Tilde{T}U^{+})^{-1}{\bf x}, \cdot \rangle: \mathbb{R}^n \rightarrow \mathbb{R}$, there exist $\{w_k\} \in C_c(\mathbb R^n)$ given by $w_k({\bf y}) = \langle {\bf S}_{\mu}^{-1}(\Tilde{T}U^{+})^{-1}{\bf x}, {\bf y} \rangle \phi_k({\bf y})$,
such that $w_k$ not only converges to $\langle {\bf S}_{\mu}^{-1}(\Tilde{T}U^{+})^{-1}{\bf x}, \cdot \rangle$ in $L^2(\mu)$, but also for each ${\bf y} \in \mathbb{R}^n$, $ |w_k({\bf y})| \leq |\langle {\bf S}_{\mu}^{-1}(\Tilde{T}U^{+})^{-1}{\bf x}, {\bf y} \rangle| $ and $ w_k({\bf y}) \rightarrow \langle {\bf S}_{\mu}^{-1}(\Tilde{T}U^{+})^{-1}{\bf x}, {\bf y} \rangle$.
Therefore, by the definition of $\Tilde{T}$ and Equation \eqref{eq:reconstruction}, we have 
\begin{equation*}
    {\bf x} = \Tilde{T}\left(\langle {\bf S}_{\mu}^{-1}(\Tilde{T}U^{+})^{-1}{\bf x}, \cdot \rangle \right) = \lim_{k \rightarrow \infty} T(w_k) = \lim_{k \rightarrow \infty} \int_{\mathbb{R}^n} w_k({\bf y}) {\bf y} d\nu({\bf y}).
\end{equation*}
Then, for any ${\bf x} \in \mathbb{R}^n$, we have
\begin{equation*}
\begin{split}
 \Vert  {\bf x} \Vert^4 &= |\langle {\bf x}, {\bf x} \rangle|^2 = \lim_{k \rightarrow \infty} \left | \int_{\mathbb{R}^n} w_k({\bf y}) \langle {\bf x}, {\bf y} \rangle d\nu({\bf y}) \right|^2 \\
   & \leq \lim_{k \rightarrow \infty} \int_{\mathbb{R}^n} |w_k({\bf y})|^2 d\nu({\bf y})  \ \int_{\mathbb{R}^n} |\langle {\bf x}, {\bf y} \rangle|^2 d\nu({\bf y})\\
   & \leq \int_{\mathbb{R}^n} |\langle {\bf S}_{\mu}^{-1}(\Tilde{T}U^{+})^{-1}{\bf x}, {\bf y} \rangle|^2 d\nu({\bf y})  \ \int_{\mathbb{R}^n} |\langle {\bf x}, {\bf y} \rangle|^2 d\nu({\bf y})\\
  &\leq \Vert {\bf S}_{\mu}^{-1} \Vert_2^2 \ \Vert(\Tilde{T}U^{+})^{-1} \Vert^2 \ \Vert {\bf x} \Vert^2  \int_{\mathbb{R}^n} \Vert {\bf y} \Vert^2 d\nu({\bf y})  \ \int_{\mathbb{R}^n} |\langle {\bf x}, {\bf y} \rangle|^2 d\nu({\bf y})\\
   &\leq \frac{M_2(\nu)}{A^2} \Big (\frac{1+\lambda_2}{1- (\lambda_1 + \frac{\delta}{\sqrt{A}})} \Big)^2   \Vert {\bf x} \Vert^2  \ \int_{\mathbb{R}^n} |\langle {\bf x}, {\bf y} \rangle|^2 d\nu({\bf y}),
\end{split}
\end{equation*}
where the first inequality is due to the Cauchy–Schwarz inequality and $\Vert {\bf S}_{\mu}^{-1} \Vert_2$ is the 2-matrix norm of ${\bf S}_{\mu}^{-1}$ satisfying $ \Vert {\bf S}_{\mu}^{-1} \Vert_2 \leq \frac{1}{A}$. In particular, we can drop the limit in the second inequality because for any $k$ and any ${\bf y} \in \mathbb{R}^n$,  $|w_k({\bf y})| \leq |\langle {\bf S}_{\mu}^{-1}(\Tilde{T}U^{+})^{-1}{\bf x}, {\bf y} \rangle|$.
Thus, for any ${\bf x} \in \mathbb{R}^n$, 
\begin{equation*}
   \frac{ A^2(1-(\lambda_1 + \frac{\delta}{\sqrt{A}}))^2}{(1+\lambda_2)^2 M_2(\nu)} \Vert  {\bf x} \Vert^2  \leq \int_{\mathbb{R}^n} |\langle {\bf x}, {\bf y} \rangle|^2 d\nu({\bf y}) \leq M_2(\nu) \ \Vert  {\bf x} \Vert^2. \qedhere
\end{equation*}
\end{proof}

The following lemma is inspired by the special case $\lambda_1=\lambda_2=0$ and $\delta=\sqrt{R}$. However, its proof only relies on the definition of probabilistic frames.

\begin{lemma}[Sweetie's Lemma]\label{sweetie}
Let $\mu$ be a probabilistic frame for $\mathbb{R}^n$ with bounds $0<A\leq B$ and $\nu \in \mathcal{P}_2(\mathbb{R}^n)$. Suppose there exists a constant $R$ with $ 0<R < A$ such that for any ${\bf x} \in \mathbb{R}^n$, 
    \begin{equation*}
     \Big   \vert \int_{\mathbb{R}^n} |\langle {\bf x}, {\bf y} \rangle|^2 d\mu({\bf y}) - \int_{\mathbb{R}^n} |\langle {\bf x}, {\bf z} \rangle|^2  d\nu({\bf z}) \Big \vert \leq R    \Vert  {\bf x} \Vert^2. 
    \end{equation*}
Then $\nu$ is a probabilistic frame for $\mathbb{R}^n$ with bounds $A-R$ and $B+R$. 
\end{lemma}

\begin{proof}
    By the assumption of the lemma, we know that for any ${\bf x} \in \mathbb{R}^n$, 
\begin{equation*}
      \int_{\mathbb{R}^n} |\langle {\bf x}, {\bf y} \rangle|^2  d\mu({\bf y}) -R \Vert  {\bf x} \Vert^2  \leq \int_{\mathbb{R}^n} |\langle {\bf x}, {\bf z} \rangle|^2 d\nu({\bf z})  \leq  \int_{\mathbb{R}^n} |\langle {\bf x}, {\bf y} \rangle|^2  d\mu({\bf y})+ R    \Vert  {\bf x} \Vert^2.
    \end{equation*}
Since $\mu$ is a probabilistic frame for $\mathbb{R}^n$ with bounds $A$ and $B$, then for any ${\bf x} \in \mathbb{R}^n$,
\begin{equation*}
      (A-R) \Vert  {\bf x} \Vert^2   \leq \int_{\mathbb{R}^n} |\langle {\bf x}, {\bf z} \rangle|^2 d\nu({\bf z})  \leq  (B + R)    \Vert  {\bf x} \Vert^2 . 
    \end{equation*} 
    Therefore, $\nu$ is a probabilistic frame for $\mathbb{R}^n$ with bounds $A-R$ and $B+R$. 
\end{proof}

\subsection{Perturbations Including Probabilistic Dual Frames}\label{TransportDual} As in the classical frame setting in Theorem \ref{PaleyDualFrames}, frame perturbation estimates can depend on mixed terms of dual frames, which motivates their probabilistic analogue. In this work, we give sufficient perturbation conditions where probabilistic dual frames are used.
\begin{theorem}\label{StabilityPDualFrame}
 Let $\mu$ be a probabilistic frame for $\mathbb{R}^n$ and $\nu \in \mathcal{P}_2(\mathbb{R}^n)$ a probabilistic dual frame of $\mu$ with respect to $\gamma_{12} \in \Gamma(\mu, \nu)$. Suppose $\eta \in \mathcal{P}_2(\mathbb{R}^n)$ and $\gamma_{23} \in \Gamma(\nu, \eta)$. Then there exists $ \Tilde{\pi} \in \mathcal{P}(\mathbb{R}^n \times \mathbb{R}^n \times \mathbb{R}^n)$ with marginals $\gamma_{12}$ and $\gamma_{23}$, and if
 \begin{equation*}
     \sigma:= \int_{\mathbb{R}^n \times \mathbb{R}^n \times \mathbb{R}^n}  \Vert {\bf y} \Vert  \Vert {\bf x}-{\bf z} \Vert d\Tilde{\pi}({\bf x},{\bf y},{\bf z}) <1,
 \end{equation*}
then $\eta$ is a probabilistic frame for $\mathbb{R}^n$ with bounds $ \frac{(1- \sigma)^2}{M_2(\nu)} \ \text{and} \  M_2(\eta)$.
If the upper frame bound for $\nu$ is $D>0$, then the frame bounds for $\eta$ are $\frac{(1- \sigma)^2}{D} \ \text{and} \  M_2(\eta)$.

 \end{theorem}
 
\begin{proof}
By the Gluing Lemma \cite[pp.~59]{figalli2021invitation}, there exists $\Tilde{\pi} \in \mathcal{P}(\mathbb{R}^n \times \mathbb{R}^n \times \mathbb{R}^n)$ such that ${\pi_{xy}}_{\#}\Tilde{\pi} = \gamma_{12} $ and $  {\pi_{yz}}_{\#}\Tilde{\pi} = \gamma_{23}$, where $\pi_{xy}$ and $\pi_{yz}$ are projections onto $({\bf x},{\bf y})$ and $({\bf y},{\bf z})$ coordinates.
Since $\eta \in \mathcal{P}_2(\mathbb{R}^n)$, $\eta$ is Bessel with bound  $ M_2(\eta)$. Next let us find the lower frame bound. Define a linear operator $L: \mathbb{R}^n \rightarrow \mathbb{R}^n$ by 
     \begin{equation*}
         L({\bf f}) = \int_{\mathbb{R}^n \times \mathbb{R}^n} \langle {\bf f}, {\bf y} \rangle {\bf z} d \gamma_{23}({\bf y},{\bf z}), \ \text{for any} \ {\bf f} \in \mathbb{R}^n.
     \end{equation*}
Since $\nu$ is a probabilistic dual frame of $\mu$ with respect to $\gamma_{12} \in \Gamma(\mu, \nu)$, then 
     \begin{equation*}
    {\bf f} = \int_{\mathbb{R}^n \times \mathbb{R}^n} \langle {\bf f}, {\bf y} \rangle {\bf x}  d\gamma_{12}({\bf x}, {\bf y}) , \ \text{for any} \ {\bf f} \in \mathbb{R}^n.
\end{equation*}
Therefore, 
\begin{equation*}
\begin{split}
    \Vert {\bf f} -L({\bf f}) \Vert 
    &= \Big \Vert \int_{\mathbb{R}^n \times \mathbb{R}^n \times \mathbb{R}^n} \langle {\bf f} , {\bf y} \rangle ({\bf x}- {\bf z}) \ d \Tilde{\pi}({\bf x}, {\bf y},{\bf z}) \Big \Vert \leq  \sigma \Vert  {\bf f} \Vert  < \Vert {\bf f}   \Vert.
\end{split}
\end{equation*}
Thus $L: \mathbb{R}^n \rightarrow \mathbb{R}^n$ is invertible and $\Vert L^{-1} \Vert \leq \frac{1}{1-\sigma} $. Then for any ${\bf f} \in \mathbb{R}^n$, 
\begin{equation*}
    {\bf f} = LL^{-1}({\bf f}) = \int_{\mathbb{R}^n \times \mathbb{R}^n} \langle L^{-1}{\bf f}, {\bf y} \rangle {\bf z} d \gamma_{23}({\bf y},{\bf z}).
\end{equation*}
Therefore, 
\begin{equation*}
\begin{split}
   \Vert  {\bf f} \Vert^4 &=  |\langle {\bf f}, {\bf f} \rangle|^2 = \Big \vert \int_{\mathbb{R}^n \times \mathbb{R}^n} \langle L^{-1}{\bf f}, {\bf y} \rangle \langle {\bf f}, {\bf z} \rangle d \gamma_{23}({\bf y},{\bf z}) \Big \vert^2 \\
    & \leq \int_{\mathbb{R}^n} \left| \langle L^{-1}{\bf f}, {\bf y} \rangle \right|^2 d\nu({\bf y})  \ \int_{\mathbb{R}^n} \left|\langle {\bf f}, {\bf z} \rangle \right|^2 d\eta({\bf z})
   \leq  \frac{D}{(1- \sigma)^2} \Vert {\bf f} \Vert^2  \int_{\mathbb{R}^n} \left|\langle {\bf f}, {\bf z} \rangle \right|^2 d\eta({\bf z}),
\end{split}
\end{equation*}
where the first inequality is due to the Cauchy-Schwarz inequality in $L^2(\gamma_{23})$ and $\gamma_{23} \in \Gamma(\nu, \eta)$, and the second inequality follows from that $\nu$ is a probabilistic frame with upper bound $D>0$. Thus, for any ${\bf f} \in \mathbb{R}^n$, 
$$ \frac{(1- \sigma)^2}{D}    \Vert  {\bf f} \Vert^2  \leq  \ \int_{\mathbb{R}^n} \left|\langle {\bf f}, {\bf z} \rangle \right|^2 d\eta({\bf z}) \leq M_2(\eta) \Vert  {\bf f} \Vert^2.
$$
Hence, $\eta$ is a probabilistic frame for $\mathbb{R}^n$ with bounds $ \frac{(1- \sigma)^2}{D} $ and $M_2(\eta)$. Since $\nu \in \mathcal{P}_2(\mathbb{R}^n)$, $M_2(\nu)$ is always an upper frame bound for $\nu$.
 In this case, the frame bounds for $\eta$  are $ \frac{(1- \sigma)^2}{M_2(\nu)} $ and $M_2(\eta)$.
\end{proof}
    
We have the following corollary using the canonical probabilistic dual frame. 
\begin{corollary}\label{dualpeturbation}
     Let $\mu$ be a probabilistic frame for $\mathbb{R}^n$ with bounds $0<A \leq B$ and $\eta \in \mathcal{P}_2(\mathbb{R}^n)$. If 
 \begin{equation*}
       \hat{\sigma}:=  \int_{\mathbb{R}^n \times \mathbb{R}^n} \Vert  {\bf S}_\mu^{-1} {\bf x}\Vert 
     \Vert {\bf x}-{\bf z} \Vert  d \mu({\bf x}) d \eta({\bf z})<1,
 \end{equation*}
 then $\eta$ is a probabilistic frame for $\mathbb{R}^n$ with bounds $A(1- \hat{\sigma})^2 $ and $ M_2(\eta)$.
\end{corollary}

\begin{proof}
Let $ \nu = {{\bf S}_\mu^{-1}}_{\#}\mu$ be the canonical probabilistic dual frame of $\mu$ with respect to $\gamma_{12}:= {({\bf Id}, {\bf S}_\mu^{-1})}_{\#}\mu \in \Gamma(\mu, {{\bf S}_\mu^{-1}}_{\#}\mu)$, and let $\gamma_{23}:= {{\bf S}_\mu^{-1}}_{\#}\mu \otimes \eta \in \Gamma({{\bf S}_\mu^{-1}}_{\#}\mu, \eta)$. Then, by the Gluing Lemma{\cite[pp.~59]{figalli2021invitation}}, one may take $\Tilde{\pi}= \gamma_{12} \otimes \eta \in \mathcal{P}(\mathbb{R}^n \times \mathbb{R}^n \times \mathbb{R}^n)$ such that ${\pi_{xy}}_{\#}\Tilde{\pi} = \gamma_{12}$ and ${\pi_{yz}}_{\#}\Tilde{\pi} = \gamma_{23}$. Hence, 
 \begin{equation*}
     \hat{\sigma}=   \int_{\mathbb{R}^n \times \mathbb{R}^n}  \Vert  {\bf S}_\mu^{-1} {\bf x}\Vert \Vert {\bf x}-{\bf z} \Vert  d \mu({\bf x}) d \eta({\bf z}) = \int_{\mathbb{R}^n \times \mathbb{R}^n \times \mathbb{R}^n}  \Vert {\bf y} \Vert \Vert {\bf x}-{\bf z} \Vert  d \Tilde{\pi}({\bf x},{\bf y},{\bf z}).
 \end{equation*}
Since $\hat{\sigma}<1$ and the upper frame bound for ${{\bf S}_\mu^{-1}}_{\#}\mu$ is $ \frac{1}{A}$,  
then by Theorem \ref{StabilityPDualFrame}, $\eta$ is a probabilistic frame for $\mathbb{R}^n$ with bounds $     A(1- \hat{\sigma})^2$ and $ M_2(\eta)$. 
\end{proof}

\begin{remark}\label{remark1}
Note that $\Vert  {\bf S}_\mu^{-1} {\bf x}\Vert \leq \Vert  {\bf S}_\mu^{-1} \Vert_2 \Vert  {\bf x} \Vert \leq \frac{1}{A} \Vert  {\bf x} \Vert$. If 
 \begin{equation*}
     \hat{ \epsilon}:=   \int_{\mathbb{R}^n \times \mathbb{R}^n} \Vert {\bf x}\Vert   \Vert {\bf x}-{\bf z} \Vert d \mu({\bf x}) d \eta({\bf z})<A,
 \end{equation*}
 then $\hat{\sigma} \leq \frac{\hat{\epsilon}}{A}<1$. Thus, $\eta$ is a probabilistic frame with bounds $A(1- \hat{\sigma})^2 $ and $ M_2(\eta)$. 
\end{remark}

Indeed, Corollary \ref{dualpeturbation} and Remark \ref{remark1} can be generalized to any $\gamma \in \Gamma(\mu, \eta)$.

\begin{proposition}\label{CoulingDualPerturbation}
Let $\mu$ be a probabilistic frame for $\mathbb{R}^n$ with bounds $0<A \leq B$ and $\eta \in \mathcal{P}_2(\mathbb{R}^n)$. Suppose there exists $\gamma \in \Gamma(\mu, \eta)$ such that
\begin{equation*}
   \epsilon:= \int_{\mathbb{R}^n \times \mathbb{R}^n}  \Vert {\bf x}\Vert  \Vert {\bf x}-{\bf z} \Vert  d \gamma({\bf x},{\bf z})<A,
\end{equation*}
then $\eta$ is a probabilistic frame for $\mathbb{R}^n$ with bounds $\frac{(A- \epsilon)^2}{B} \ \text{and} \ M_2(\eta)$,
and if 
\begin{equation*}
   \chi := \int_{\mathbb{R}^n \times \mathbb{R}^n} \Vert {\bf S}_\mu^{-1} {\bf x}\Vert   \Vert {\bf x}-{\bf z} \Vert  d \gamma({\bf x},{\bf z})<1,
\end{equation*}
then $\eta$ is a probabilistic frame for $\mathbb{R}^n$ with bounds $\frac{A^2 (1-\chi)^2}{B} \ \text{and} \ M_2(\eta)$.
\end{proposition}

\begin{proof}
Since $\eta \in \mathcal{P}_2(\mathbb{R}^n)$, $\eta$ is Bessel with bound  $M_2(\eta)$.  Using the reconstruction formula for the canonical dual frame in Equation \eqref{reconstruction2} and $\gamma \in \Gamma(\mu, \eta)$, we have
$$ 
{\bf f} =\int_{\mathbb{R}^n} \langle {\bf S}_\mu^{-1} {\bf f}, {\bf x} \rangle {\bf x}  d\mu({\bf x}) = \int_{\mathbb{R}^n \times \mathbb{R}^n} \langle {\bf S}_\mu^{-1} {\bf f}, {\bf x} \rangle {\bf x} d \gamma({\bf x},{\bf z})  , \ \text{for any} \ {\bf f} \in \mathbb{R}^n.
$$
Now define $L: \mathbb{R}^n \rightarrow \mathbb{R}^n$ by 
     \begin{equation*}
         L({\bf f}) = \int_{\mathbb{R}^n \times \mathbb{R}^n} \langle {\bf S}_\mu^{-1} {\bf f}, {\bf x} \rangle {\bf z}  d \gamma({\bf x},{\bf z}), \ \text{for any} \ {\bf f} \in \mathbb{R}^n.
     \end{equation*}
Therefore, 
\begin{equation*}
    \Vert {\bf f} -L({\bf f}) \Vert 
    = \Big \Vert \int_{\mathbb{R}^n \times \mathbb{R}^n} \langle {\bf S}_\mu^{-1}{\bf f}, {\bf x} \rangle ({\bf x}-{\bf z}) d \gamma({\bf x}, {\bf z}) \Big \Vert 
     \leq \epsilon \ \Vert {\bf S}_\mu^{-1} \Vert_2 \ \Vert  {\bf f} \Vert  \leq \frac{\epsilon}{A}  \Vert  {\bf f} \Vert < \Vert  {\bf f} \Vert,
\end{equation*}
where the second inequality follows from that $\|{\bf S}_\mu^{-1}\|_2 \leq \frac{1}{A}$. 
Thus, $L: \mathbb{R}^n \rightarrow \mathbb{R}^n$ is invertible and $\Vert L^{-1} \Vert \leq \frac{1}{1-\epsilon/ A}$.
Then, for any ${\bf f} \in \mathbb{R}^n$, 
\begin{equation*}
    {\bf f} = LL^{-1}({\bf f}) =  \int_{\mathbb{R}^n \times \mathbb{R}^n} \langle {\bf S}_\mu^{-1}L^{-1} {\bf f}, {\bf x} \rangle {\bf z} d\gamma({\bf x},{\bf z}).
\end{equation*}
Therefore, 
\begin{equation*}
\begin{split}
   \Vert {\bf f} \Vert^4 &=  |\langle {\bf f}, {\bf f} \rangle|^2 = \Big \vert \int_{\mathbb{R}^n \times \mathbb{R}^n} \langle {\bf S}_\mu^{-1}L^{-1} {\bf f}, {\bf x} \rangle \langle  {\bf f}, {\bf z} \rangle d \gamma({\bf x},{\bf z}) \Big \vert^2 \\
    & \leq \int_{\mathbb{R}^n} \left|\langle {\bf S}_\mu^{-1} L^{-1}{\bf f}, {\bf x} \rangle \right|^2 d\mu({\bf x})  \ \int_{\mathbb{R}^n} \left| \langle {\bf f}, {\bf z} \rangle \right|^2 d\eta({\bf z})\\
 &\leq B \Vert  {\bf S}_\mu^{-1} \Vert_2^2  \ \Vert L^{-1} \Vert^2  \ \Vert {\bf f} \Vert^2   \int_{\mathbb{R}^n} \left|\langle {\bf f}, {\bf z} \rangle \right|^2 d\eta({\bf z}) \leq \frac{B\Vert {\bf f} \Vert^2}{A^2 (1-\frac{\epsilon}{A})^2 }  \int_{\mathbb{R}^n} \left| \langle {\bf f}, {\bf z} \rangle \right |^2 d\eta({\bf z}),
\end{split}
\end{equation*}
where the first inequality is due to the Cauchy-Schwarz inequality in $L^2(\gamma)$ and $\gamma \in \Gamma(\mu, \eta)$, and the second inequality follows from that $\mu$ is a probabilistic frame with upper bound $B$. 
Then, for any ${\bf f} \in \mathbb{R}^n$, 
$$ \frac{(A- \epsilon)^2}{B}    \Vert  {\bf f} \Vert^2  \leq   \int_{\mathbb{R}^n} |\langle {\bf f}, {\bf z} \rangle|^2 d\eta({\bf z}) \leq M_2(\eta) \Vert  {\bf f} \Vert^2.$$
Thus, $\eta$ is a probabilistic frame with bounds $\frac{(A- \epsilon)^2}{B}$ and $M_2(\eta)  $.
Furthermore, if 
\begin{equation*}
   \chi := \int_{\mathbb{R}^n \times \mathbb{R}^n}  \Vert {\bf S}_\mu^{-1} {\bf x}\Vert \Vert {\bf x}-{\bf z} \Vert  d \gamma({\bf x},{\bf z})<1,
\end{equation*}
then 
 \begin{equation*}
    \Vert {\bf f} -L({\bf f}) \Vert \leq  \int_{\mathbb{R}^n \times \mathbb{R}^n} |\langle {\bf f}, {\bf S}_\mu^{-1} {\bf x} \rangle| \ \|{\bf x}-{\bf z}\|  d \gamma({\bf x}, {\bf z}) \leq \chi \ \Vert  {\bf f} \Vert < \Vert  {\bf f} \Vert. 
\end{equation*}
Therefore, $L$ is invertible and $\Vert L^{-1} \Vert \leq \frac{1}{1-\chi}$.
Similar arguments show that $\eta$ is a probabilistic frame for $\mathbb{R}^n$ with bounds $\frac{A^2 (1-\chi)^2}{B} \ \text{and} \ M_2(\eta)$.
 \end{proof}

The key step in the above proof is to use the canonical probabilistic dual frame to reconstruct the vector ${\bf f}$. An alternative approach uses the canonical probabilistic Parseval frame, by which we obtain the last proposition in this paper. 
\begin{proposition}\label{ParsevelConstructionPerturbation}
Let $\mu$ be a probabilistic frame for $\mathbb{R}^n$ with bounds $0<A \leq B$ and $\eta \in \mathcal{P}_2(\mathbb{R}^n)$. Suppose there exists $\gamma \in \Gamma(\mu, \eta)$ such that 
\begin{equation*}
   \tau:= \int_{\mathbb{R}^n  \times \mathbb{R}^n} \Vert {\bf x} \Vert \Vert  {\bf S}_\mu^{-1/2} {\bf x}- {\bf z}  \Vert d \gamma({\bf x}, {\bf z}) < \sqrt{A},
\end{equation*}
then $\eta$ is a probabilistic frame for $\mathbb{R}^n$  with bounds $  \frac{(\sqrt{A}- \tau)^2}{B} \ \text{and} \ M_2(\eta)$.
\end{proposition}
\begin{proof}
Since $\eta \in \mathcal{P}_2(\mathbb{R}^n)$, $\eta$ is Bessel with bound  $M_2(\eta)$. By the reconstruction formula of the canonical Parseval frame in Equation \eqref{reconstruction1}, for any ${\bf f} \in \mathbb{R}^n$,
\begin{equation*}
   {\bf f}=  \int_{\mathbb{R}^n} \langle {\bf S}_\mu^{-1/2} {\bf f}, {\bf x} \rangle {\bf S}_\mu^{-1/2} {\bf x}  \ d\mu({\bf x}) = \int_{\mathbb{R}^n \times \mathbb{R}^n} \langle {\bf S}_\mu^{-1/2} {\bf f}, {\bf x} \rangle {\bf S}_\mu^{-1/2} {\bf x}  \ d\gamma({\bf x}, {\bf z}).
\end{equation*}
Define $L: \mathbb{R}^n \rightarrow \mathbb{R}^n$ by 
$$
L({\bf f}) = \int_{\mathbb{R}^n \times \mathbb{R}^n} \langle {\bf S}_\mu^{-1/2}{\bf f}, {\bf x} \rangle {\bf z} d \gamma({\bf x},{\bf z}), \ \text{for any} \ {\bf f} \in \mathbb{R}^n.
$$
Therefore, $L$ is linear and 
\begin{equation*}
    \Vert {\bf f} -L({\bf f}) \Vert \leq \Vert {\bf S}_\mu^{-1/2} {\bf f}  \Vert \int_{\mathbb{R}^n  \times \mathbb{R}^n} \Vert {\bf x} \Vert \Vert  {\bf S}_\mu^{-1/2} {\bf x}- {\bf z}  \Vert d \gamma({\bf x}, {\bf z})  \leq \frac{\tau}{\sqrt{A}}  \Vert  {\bf f} \Vert <  \Vert  {\bf f} \Vert.
\end{equation*}
Thus, $L: \mathbb{R}^n \rightarrow \mathbb{R}^n$ is invertible and $\Vert L^{-1} \Vert \leq \frac{1}{1-\tau / \sqrt{A}}$.
Then for any ${\bf f} \in \mathbb{R}^n$, 
\begin{equation*}
    {\bf f} = LL^{-1}({\bf f}) =  \int_{\mathbb{R}^n \times \mathbb{R}^n} \langle {\bf S}_\mu^{-1/2}L^{-1} {\bf f}, {\bf x} \rangle {\bf z} d \gamma({\bf x},{\bf z}).
\end{equation*}
By the Cauchy-Schwarz inequality and the frame property of $\mu$, we have
\begin{equation*}
\begin{split}
   \Vert {\bf f} \Vert^4 &=  |\langle {\bf f}, {\bf f} \rangle|^2 = \Big \vert \int_{\mathbb{R}^n \times \mathbb{R}^n} \langle {\bf S}_\mu^{-1/2}L^{-1} {\bf f}, {\bf x} \rangle \langle  {\bf f}, {\bf z} \rangle d \gamma({\bf x},{\bf z})  \Big  \vert^2 \\
    & \leq \int_{\mathbb{R}^n} |\langle {\bf S}_\mu^{-1/2} L^{-1}{\bf f}, {\bf x} \rangle|^2 d\mu({\bf x})  \ \int_{\mathbb{R}^n} |\langle {\bf f}, {\bf z} \rangle|^2 d\eta({\bf z}) \leq \frac{B \Vert {\bf f} \Vert^2}{A (1-\frac{\tau}{\sqrt{A}})^2 }   \int_{\mathbb{R}^n} |\langle {\bf f}, {\bf z} \rangle|^2 d\eta({\bf z}).
\end{split}
\end{equation*}
Then for ${\bf f} \in \mathbb{R}^n$, 
 $$ 
 \frac{(\sqrt{A}-\tau)^2}{B}    \Vert  {\bf f} \Vert^2  \leq  \ \int_{\mathbb{R}^n} |\langle {\bf f}, {\bf z} \rangle|^2 d\eta({\bf z}) \leq M_2(\eta) \ \Vert  {\bf f} \Vert^2.
 $$ 
 Therefore, $\eta$ is a probabilistic frame for $\mathbb{R}^n$ with bounds $\frac{(\sqrt{A}- \tau)^2}{B}  \ \text{and} \ M_2(\eta)$.
\end{proof}

\section{Proof of Proposition \ref{QuadClose} }\label{sec:proof}
\begin{proof}[Proof of Proposition \ref{QuadClose}]
Since  $\nu \in \mathcal{P}_2(\mathbb{R}^n)$, $\nu$ is Bessel with bound  $M_2(\nu)$. Next, let us show the lower frame bound. 
Using the reconstruction formula of the canonical dual frame in Equation \eqref{reconstruction2} and  the fact that $\gamma \in \Gamma(\mu, \nu)$, we have
\begin{equation*}
   {\bf f}=  \int_{\mathbb{R}^n} \langle {\bf f},  {\bf S}_\mu^{-1} {\bf x} \rangle {\bf x}   d\mu({\bf x}) = \int_{\mathbb{R}^n \times \mathbb{R}^n} \langle {\bf f},  {\bf S}_\mu^{-1} {\bf x} \rangle {\bf x}  d\gamma({\bf x},{\bf y}), \ \text{for any} \ {\bf f} \in \mathbb{R}^n.
\end{equation*}
Now define a linear map $L: \mathbb{R}^n \rightarrow \mathbb{R}^n$ by 
$$ L({\bf f}) = \int_{\mathbb{R}^n \times \mathbb{R}^n} \langle {\bf f},  {\bf S}_\mu^{-1} {\bf x} \rangle  {\bf y}  d\gamma({\bf x},{\bf y}), \ \text{for any ${\bf f} \in \mathbb{R}^n$}.
$$
By the Cauchy-Schwarz inequality and the fact that $ {{\bf S}_\mu^{-1}}_{\#} \mu$ is a probabilistic frame with upper bound $\frac{1}{A}$, we have
\begin{equation*}
\begin{split}
    \Vert {\bf f} -L({\bf f}) \Vert^2 
    &\leq \Big | \int_{\mathbb{R}^n \times \mathbb{R}^n} \vert \langle {\bf f}, {\bf S}_\mu^{-1} {\bf x} \rangle \vert  \ \Vert {\bf x}-{\bf y} \Vert d \gamma({\bf x}, {\bf y})  \Big |^2 \\
    &\leq \int_{\mathbb{R}^n} \vert \langle {\bf f}, {\bf x} \rangle \vert^2  d({{\bf S}_\mu^{-1}}_{\#}\mu) ({\bf x})  \int_{\mathbb{R}^n  \times \mathbb{R}^n} \Vert {\bf x} -{\bf y}\Vert^2 d \gamma({\bf x}, {\bf y})   \leq  \frac{\lambda}{A}  \Vert  {\bf f} \Vert ^2 < \Vert  {\bf f} \Vert ^2,  
\end{split}
\end{equation*}
where we use the pushforward (change-of-variables) identity for $ {{\bf S}_\mu^{-1}}_{\#} \mu$ in the second inequality.  Thus, $L: \mathbb{R}^n \rightarrow \mathbb{R}^n$ is invertible and $\Vert L^{-1} \Vert \leq \frac{1}{1-\sqrt{\frac{\lambda}{A}}}$.
Then, for any ${\bf f} \in \mathbb{R}^n$, 
\begin{equation*}
    {\bf f} = LL^{-1}({\bf f}) =  \int_{\mathbb{R}^n \times \mathbb{R}^n} \langle L^{-1}{\bf f},  {\bf S}_\mu^{-1} {\bf x} \rangle  {\bf y}  d\gamma({\bf x},{\bf y}).
\end{equation*}
By the Cauchy-Schwarz inequality and the frame property of $ {{\bf S}_\mu^{-1}}_{\#} \mu$ again, we have
\begin{equation*}
\begin{split}
   \Vert {\bf f} \Vert^4 &=  |\langle {\bf f}, {\bf f}\rangle|^2 = \Big \vert  \int_{\mathbb{R}^n \times \mathbb{R}^n} \langle L^{-1}{\bf f},  {\bf S}_\mu^{-1} {\bf x} \rangle   \langle  {\bf f}, {\bf y} \rangle   d\gamma({\bf x},{\bf y}) \Big \vert^2\\
   & \leq \int_{\mathbb{R}^n} |\langle L^{-1}{\bf f},  {\bf S}_\mu^{-1} {\bf x} \rangle|^2 d\mu({\bf x})  \ \int_{\mathbb{R}^n} |\langle {\bf f}, {\bf y} \rangle|^2 d\nu({\bf y})\\
 &\leq \frac{ \Vert L^{-1} {\bf f} \Vert^2 }{A}    \ \int_{\mathbb{R}^n} |\langle {\bf f}, {\bf y} \rangle|^2 d\nu({\bf y})
   \leq \frac{\Vert {\bf f} \Vert^2 }{A (1-\sqrt{\frac{\lambda}{A}})^2}   \int_{\mathbb{R}^n} |\langle {\bf f}, {\bf y} \rangle|^2 d\nu({\bf y}).
\end{split}
\end{equation*}
Then for any ${\bf f} \in \mathbb{R}^n$, 
 $$ 
 (\sqrt{A}-\sqrt{\lambda})^2    \Vert  {\bf f} \Vert^2  \leq  \ \int_{\mathbb{R}^n} |\langle {\bf f}, {\bf y} \rangle|^2 d\nu({\bf y}) \leq M_2(\nu) \Vert {\bf f}  \Vert^2. 
 $$  
 Therefore, $\nu$ is a probabilistic frame for $\mathbb{R}^n$ with bounds $(\sqrt{A}-\sqrt{\lambda})^2 \ \text{and} \ M_2(\nu)$.
\end{proof}

\section*{Acknowledgements}
This paper is dedicated to my family and friends.
The author thanks the reviewer for their constructive comments, which have greatly improved this manuscript.
\bibliographystyle{unsrt} 
\bibliography{refs_jmaa}

@book{paley1934fourier,
  title={Fourier transforms in the complex domain},
  author={Paley, Raymond E. A. C. and Wiener, Norbert},
  volume={19},
  year={1934},
  publisher={American Mathematical Society}
}

@article{wickman2017duality,
  title={Duality and geodesics for probabilistic frames},
  author={Wickman, Clare G. and Okoudjou, Kasso A.},
  journal={Linear Algebra and Its Applications},
  volume={532},
  pages={198--221},
  year={2017},
  publisher={Elsevier}
}

@article{boas1940general,
  title={General expansion theorems},
  author={Boas Jr, R. P. },
  journal={Proceedings of the National Academy of Sciences},
  volume={26},
  number={2},
  pages={139--143},
  year={1940},
  publisher={National Acad Sciences}
}

@book{young2001introduction,
  title={An Introduction to Non-Harmonic Fourier Series, Revised Edition, 93},
  author={Young, Robert M.},
  year={2001},
  publisher={Elsevier}
}

@article{christensen1995paley,
  title={A {P}aley--{W}iener theorem for frames},
  author={Christensen, Ole},
  journal={Proceedings of the American Mathematical Society},
  volume={123},
  number={7},
  pages={2199--2201},
  year={1995}
}

@article{christensen1995frame,
  title={Frame perturbations},
  author={Christensen, Ole},
  journal={Proceedings of the American Mathematical Society},
  volume={123},
  number={4},
  pages={1217--1220},
  year={1995}
}

@article{christensen1997perturbations,
  title={Perturbations of {B}anach frames and atomic decompositions},
  author={Christensen, Ole and Heil, Christopher},
  journal={Mathematische Nachrichten},
  volume={185},
  number={1},
  pages={33--47},
  year={1997},
  publisher={Wiley Online Library}
}

@article{cazassa1997perturbation,
  title={Perturbation of operators and applications to frame theory},
  author={Cazassa, Peter G. and Christensen, Ole},
  journal={Journal of Fourier Analysis and Applications},
  volume={3},
  number={5},
  pages={543--557},
  year={1997},
  publisher={Springer}
}

@article{sun2007stability,
  title={Stability of g-frames},
  author={Sun, Wenchang},
  journal={Journal of Mathematical Analysis and Applications},
  volume={326},
  number={2},
  pages={858--868},
  year={2007},
  publisher={Elsevier}
}

@article{chen2014perturbations,
  title={Perturbations of frames},
  author={Chen, Dongyang and Li, Lei and Zheng, Bentuo},
  journal={Acta Mathematica Sinica, English Series},
  volume={30},
  number={7},
  pages={1089--1108},
  year={2014},
  publisher={Springer}
}

@article{christensen2017operator,
  title={Operator representations of frames: boundedness, duality, and stability},
  author={Christensen, Ole and Hasannasab, Marzieh},
  journal={Integral Equations and Operator Theory},
  volume={88},
  pages={483--499},
  year={2017},
  publisher={Springer}
}

@article{christensen2016introduction,
  title={An Introduction to Frames and {R}iesz Bases},
  author={Christensen, Ole},
  journal={Applied and Numerical Harmonic Analysis},
  year={2016},
  publisher={Springer International Publishing}
}

@article{duffin1952class,
  title={A class of nonharmonic {F}ourier series},
  author={Duffin, Richard J. and Schaeffer, Albert C.},
  journal={Transactions of the American Mathematical Society},
  volume={72},
  number={2},
  pages={341--366},
  year={1952}
}

@article{casazza2013kadison,
  title={The {K}adison--{S}inger and {P}aulsen problems in finite frame theory},
  author={Casazza, Peter G.},
  journal={Finite frames: theory and applications},
  pages={381--413},
  year={2013},
  publisher={Springer}
}

@book{grochenig2001foundations,
  title={Foundations of time-frequency analysis},
  author={Gr{\"o}chenig, Karlheinz},
  year={2001},
  publisher={Springer Science \& Business Media}
}

@book{daubechies1992ten,
  title={Ten lectures on wavelets},
  author={Daubechies, Ingrid},
  year={1992},
  publisher={SIAM}
}

@article{ehler2012random,
  title={Random tight frames},
  author={Ehler, Martin},
  journal={Journal of Fourier Analysis and Applications},
  volume={18},
  number={1},
  pages={1--20},
  year={2012},
  publisher={Springer}
}

@article{ehler2011frame,
  title={Frame theory in directional statistics},
  author={Ehler, Martin and Galanis, Jennifer},
  journal={Statistics \& probability letters},
  volume={81},
  number={8},
  pages={1046--1051},
  year={2011},
  publisher={Elsevier}
}

@article{ehler2012minimization,
  title={Minimization of the probabilistic p-frame potential},
  author={Ehler, Martin and Okoudjou, Kasso A.},
  journal={Journal of Statistical Planning and Inference},
  volume={142},
  number={3},
  pages={645--659},
  year={2012},
  publisher={Elsevier}
}

@article{ehler2013probabilistic,
  title={Probabilistic frames: an overview},
  author={Ehler, Martin and Okoudjou, Kasso A.},
  journal={Finite frames},
  pages={415--436},
  year={2013},
  publisher={Springer}
}

@phdthesis{wickman2014optimal,
  title={An optimal transport approach to some problems in frame theory},
  author={Wickman, Clare G.},
  year={2014},
  school={University of Maryland, College Park}
}

@article{wickman2023gradient,
  title={Gradient Flows for Probabilistic Frame Potentials in the {W}asserstein Space},
  author={Wickman, Clare G. and Okoudjou, Kasso A.},
  journal={SIAM Journal on Mathematical Analysis},
  volume={55},
  number={3},
  pages={2324--2346},
  year={2023},
  publisher={SIAM}
}

@article{maslouhi2019probabilistic,
  title={Probabilistic tight frames and representation of positive operator-valued measures},
  author={Maslouhi, M. and Loukili, S.},
  journal={Applied and Computational Harmonic Analysis},
  volume={47},
  number={1},
  pages={212--225},
  year={2019},
  publisher={Elsevier}
}

@article{cheng2019optimal,
  title={Optimal properties of the canonical tight probabilistic frame},
  author={Cheng, Desai and Okoudjou, Kasso A.},
  journal={Numerical Functional Analysis and Optimization},
  volume={40},
  number={2},
  pages={216--240},
  year={2019},
  publisher={Taylor \& Francis}
}

@article{loukili2020minimization,
  title={A minimization problem for probabilistic frames},
  author={Loukili, S. and Maslouhi, M.},
  journal={Applied and Computational Harmonic Analysis},
  volume={49},
  number={2},
  pages={558--572},
  year={2020},
  publisher={Elsevier}
}

@book{figalli2021invitation,
  title={An Invitation to Optimal Transport, Wasserstein Distances, and Gradient Flows},
  author={Figalli, Alessio and Glaudo, Federico},
  year={2021},
  publisher={EMS Press}
}

@article{chen2025probabilistic,
  title={Probabilistic frames and {W}asserstein distances},
  author={Chen, Dongwei and Schmoll, Martin},
  journal={arXiv preprint arXiv:2501.02602},
  year={2025}
}
\end{document}